\documentclass[graybox]{svmult}

\usepackage{type1cm}
\usepackage{makeidx}
\usepackage{graphicx}
\usepackage{multicol}
\usepackage[bottom]{footmisc}
\usepackage{newtxtext}       
\usepackage{newtxmath}       
\usepackage{tikz}
\usepackage{pgfplots}
\pgfplotsset{compat=newest}

\begin{document}

\title*{On energy preserving high-order discretizations for nonlinear acoustics}
\author{Herbert Egger and Vsevolod Shashkov}
\institute{Herbert Egger \at TU Darmstadt, \email{egger@mathematik.tu-darmstadt.de}
\and Vsevolod Shashkov \at TU Darmstadt \email{shashkov@mathematik.tu-darmstadt.de}}
\maketitle

\abstract*{This paper addresses the numerical solution of the Westervelt equation, which arises as one of the model equations in nonlinear acoustics. The problem is rewritten in a canonical form that allows the systematic discretization by Galerkin approximation in space and time. Exact energy preserving methods of formally arbitrary order are obtained and their efficient realization as well as the relation to other frequently used methods is discussed.}

\abstract{This paper addresses the numerical solution of the Westervelt equation, which arises as one of the model equations in nonlinear acoustics. The problem is rewritten in a canonical form that allows the systematic discretization by Galerkin approximation in space and time. Exact energy preserving methods of formally arbitrary order are obtained and their efficient realization as well as the relation to other frequently used methods is discussed.}

\section{Introduction}

The modeling of nonlinear effects arising in the presence of high intensity acoustic fields is one of the central subjects of nonlinear acoustics \cite{HamiltonBlackstock}. One widely used model in this area is the Westervelt equation \cite{Kaltenbacher09,Westervelt63} which in dimensionless form can be written as
\begin{equation}\label{eq:westervelt}
\partial_{tt} \psi - \Delta \psi = \alpha \Delta (\partial_t \psi) + \beta \partial_t (\partial_t \psi)^2. 
\end{equation}
The two terms on the right hand side, scaled with constants $\alpha,\beta \ge 0$, account for viscous and nonlinear effects of the medium and constitute the deviations from the standard linear wave equation. 
Equation \eqref{eq:westervelt} is written here in terms of the velocity potential $\psi$ which is related to the acoustic velocity and pressure variations by 
\begin{equation} \label{eq:vp}
v = -\nabla \psi \qquad \text{and} \qquad p = \partial_t \psi.  
\end{equation}
Similar to the linear wave equation, the Westervelt equation also encodes the principle of energy conservation. Using \eqref{eq:vp}, the dimensionless acoustic energy contained 
in a bounded domain $\Omega$ can be expressed in terms of the velocity potential by
\begin{equation} \label{eq:energy}
\mathcal{E}(\psi,\partial_t \psi) = \int_\Omega \tfrac{1}{2} |\nabla \psi|^2 + \left(\tfrac{1}{2} - \tfrac{2\beta}{3} \partial_t \psi\right) |\partial_t \psi|^2 dx  
\end{equation}
One can verify by elementary computations that solutions of \eqref{eq:westervelt}, when complemented, e.g., by homogeneous boundary conditions $\partial_n \psi = 0$, satisfy 
\begin{equation} \label{eq:identity}
\frac{d}{dt} \mathcal{E}(\psi,\partial_t \psi) = -\alpha \int_\Omega |\nabla (\partial_t \psi)|^2 dx. 
\end{equation}
This \emph{energy identity} states that in a closed system the acoustic energy is conserved exactly up to dissipation caused by viscous effects. For $\alpha \ge 0$, the Westervelt equation \eqref{eq:westervelt} thus models a passive system.
This property is of fundamental importance not only for the analysis of the problem \cite{Kaltenbacher09} but also for the accuracy and long-term stability of discretization schemes; see \cite{LeimkuhlerReich} and the references given there. 

Various discretization schemes for the linear wave equation can be extended to nonlinear acoustics. Among the most widely used approaches are the finite-difference-time-domain method \cite{Hallaj99,Karamalis10,Okita11}, finite-volume schemes \cite{Fagnan08,Velasco15}, and finite-element methods together with Newmark time-stepping \cite{Cohen02,Hoffelner01,Tsuchiya92}. 
To the best of our knowledge, none of the mentioned approaches is capable to exactly reproduce the energy identity \eqref{eq:identity} on the discrete level in the presence of nonlinearities. 

In this paper, we propose a systematic strategy for the high-order approximation of nonlinear acoustics in space and time which exactly satisfies an integral version of the energy identity \eqref{eq:identity} on the discrete level. Our approach utilizes the fact that the Westervelt equation \eqref{eq:westervelt} can be written as a generalized gradient system 
\begin{equation} \label{eq:gradient}
\mathcal{C}(u) \partial_t u = - \mathcal{H}'(u) 
\end{equation}
with $u=(\psi,\partial_t \psi)$ denoting the state and $\mathcal{H}(u)=\mathcal{E}(\psi,\partial_t\psi)$ the energy of the system.
The energy identity \eqref{eq:identity} is then a direct consequence of the particular structure of this system; see below. 
As illustrated in \cite{Egger18}, the structure-preserving discretization of \eqref{eq:gradient} can be obtained in a systematic manner by Galerkin approximation in space and time. 
For the space discretization, we utilize a finite-element approximation with mass-lumping. The time-integration resulting from our approach can be interpreted as a variant of particular Runge-Kutta methods and is strongly related to discrete gradient and average vector field collocation methods \cite{Gonzales96,HairerLubich14,McLachlan99}. 

The remainder of the manuscript is organized as follows: 
In Section~\ref{sec:canonical}, we rewrite the Westervelt equation \eqref{eq:westervelt} into the non-standard canonical form \eqref{eq:gradient}. 
Our discretization strategy is then introduced in Section~\ref{sec:disc}, and we show that the energy identity remains valid after discretization.
In Section~\ref{sec:impl}, we briefly discuss some details of the numerical realization and the connection to other discretization methods. In Section~\ref{sec:num}, we illustrate the exact energy-conservation in the absence of viscous effects for one-dimensional example.

\section{A canonical form of the Westervelt equation} \label{sec:canonical}

We introduce $p=\partial_t \psi$ as new variable and write $u=(\psi,p)$ and $\mathcal{H}(u)=\mathcal{E}(\psi,p)$.
The derivative $\mathcal{H}'(u)$ of the energy in direction $v=(\eta,q)$ is then given by
\begin{eqnarray*}
\langle \mathcal{H}'(u),v\rangle 
= \langle \mathcal{E}'(\psi,p), (\eta,q) \rangle 
= \int_\Omega \nabla \psi \cdot \nabla \eta + (1-2\beta p) p \cdot q \, dx.
\end{eqnarray*}
Using integration-by-parts for the first term under the integral and homogeneous boundary conditions $\partial_n \psi = 0$ on $\partial \Omega$, we can now formally represent the negative derivative of the energy functional as a two-component function
\begin{equation} \label{eq:derivative}
-\mathcal{H}'(u) = \left(\Delta \psi, -(1-2\beta p) p\right).
\end{equation}
In order to bring equation \eqref{eq:westervelt} into the canonical form \eqref{eq:gradient}, we should thus derive an equivalent first order system with right hand sides given by $-\mathcal{H}'(u)$. 
By elementary computations, one can verify the following statements. 
\begin{lemma} \label{lem:equiv}
The Westervelt equation \eqref{eq:westervelt} is equivalent to the system \begin{eqnarray}
(1-2\beta p) \partial_t p - \alpha \Delta \partial_t \psi &=& \Delta \psi. \label{eq:canonical1}\\
-(1-2\beta p) \partial_t \psi &=& - (1-2\beta p) p.  \label{eq:canonical2} \\
\notag 
\end{eqnarray}
\vspace*{-3em}
\end{lemma}
\begin{proof}
Differentiating the last term in \eqref{eq:westervelt} yields 
\begin{equation*}
\beta \partial_t (\partial_t \psi)^2 = 2 \beta (\partial_t \psi) \partial_{tt} \psi. 
\end{equation*}
Using this identiy and a slight rearrangment of terms, the Westervelt equation can thus be rewritten equivalently as 
\begin{equation*}
(1 - 2\beta \partial_t \psi) \partial_{tt} \psi - \alpha \Delta (\partial_t \psi) = \Delta \psi. 
\end{equation*}
By replacing $\partial_t \psi$ and $\partial_{tt}\psi$ in the first term by $p$ and $\partial_t p$, we already obtain \eqref{eq:canonical1}. 
The second equation \eqref{eq:canonical2} is an immediate concequence of the identity $p = \partial_t \psi$. 
\end{proof}
\begin{remark}
Abbreviating $u=(\psi,p)$ and $\mathcal{H}(u)=\mathcal{E}(\psi,p)$ as above, the system \eqref{eq:canonical1}--\eqref{eq:canonical2} can be seen to formally be in the canonical form \eqref{eq:gradient} with 
\begin{equation*}
\mathcal{C}(u)=\begin{pmatrix} -\alpha \Delta & (1-2\beta p) \\ -(1-2\beta p) & 0 \end{pmatrix}.  
\end{equation*}
The somewhat unconventional form of the system \eqref{eq:canonical1}--\eqref{eq:canonical2} is dictated by the underlying energy, whose derivative has to appear in the right hand side of the equations. 
\end{remark}
Our discretization will be based on the following weak formulation of \eqref{eq:canonical1}--\eqref{eq:canonical2}. 
\begin{lemma} \label{lem:weak}
Let $(\psi,p)$ denote a smooth solution of the system \eqref{eq:canonical1}--\eqref{eq:canonical2} on $\Omega$ with homogeneous boundary values $\partial_n \psi=0$ on $\partial\Omega$ for $a \le t \le b$. Then 
\begin{eqnarray}
\langle (1-2\beta p(t)) \partial_t p(t), \eta\rangle + \alpha \langle \nabla \partial_t \psi(t), \nabla \eta\rangle &=& -\langle \nabla \psi(t), \nabla \eta\rangle \label{eq:weak1}\\
-\langle (1-2\beta p(t)) \partial_t \psi(t), q \rangle &=& -\langle(1-2\beta p(t)) p(t), q\rangle  \label{eq:weak2}\\
\notag
\end{eqnarray}
\vspace*{-3em}

\noindent
for all test functions $\eta,q \in H^1(\Omega)$ and all $a \le t \le b$.
The bracket $\langle u,v\rangle = \int_\Omega u v \, dx$ is used here to denote the scalar product on $L^2(\Omega)$.
\end{lemma}
\begin{proof}
The two identities follow by multiplying \eqref{eq:canonical1}--\eqref{eq:canonical2} with appropriate test functions, integrating over $\Omega$, and integration-by-parts for the terms with the Laplacian. The boundary terms vanish due to the homogeneous boundary conditions. 
\end{proof}
We now show that the energy identity \eqref{eq:identity} follows directly from this weak formulation. 
\begin{lemma} \label{lem:energy}
Let $(\psi,p)$ denote a solution of the weak formulation \eqref{eq:weak1}--\eqref{eq:weak2}. Then
\begin{equation*}
\frac{d}{dt} \mathcal{E}(\psi(t),p(t)) = -\alpha \int_\Omega |\nabla (\partial_t \psi(t))|^2 dx. 
\end{equation*}
\end{lemma}
\begin{proof}
Formal differentiation of the energy yields 
\begin{eqnarray*}
\frac{d}{dt} \mathcal{E}(\psi,p) 
&=& \langle \mathcal{E}'(\psi,p), (\partial_t \psi,\partial_t p)\rangle \\
&=& \langle \nabla \psi, \nabla \partial_t \psi\rangle + \langle (1-2\beta p) p, \partial_t p\rangle,
\end{eqnarray*}
where we used the representation of the energy derivative derived above. 
The two terms correspond to the right hand sides of the weak formulation 
\eqref{eq:weak1}--\eqref{eq:weak2} with test functions $\eta=\partial_t \psi$ and $q=\partial_t p$. Using the weak formulation, we thus obtain
\begin{eqnarray*}
\frac{d}{dt} \mathcal{E}(\psi,p)
&=& -\langle (1-2\beta p) \partial_t p, \partial_t \psi\rangle - \alpha \langle \nabla \partial_t \psi, \nabla \partial_t \psi \rangle + \langle (1-2\beta p) \partial_t \psi, \partial_t p \rangle.
\end{eqnarray*}
Now the first and last term on the right hand side cancel out and the assertion follows by noting that $\langle \nabla \partial_t \psi,\nabla\partial_t \psi\rangle=\int_\Omega |\nabla\partial_t \psi|^2 dx $ by definition of the bracket. 
\end{proof}
\begin{remark}
The proof of the previous lemma reveals that the energy identity \eqref{eq:identity} is a direct consequence already of the particular structure of the weak formulation \eqref{eq:weak1}--\eqref{eq:weak2}. Since this form is preserved automatically under projection, one can obtain a structure preserving discretization by Galerkin approximation; see \cite{Egger18} for details. In the following section, we discuss a particular approximation based on finite elements.
\end{remark}

\section{Structure-preserving discretization} \label{sec:disc}

Let $\mathcal{T}_h = \{K\}$ denote a mesh, i.e., a geometrically conforming and uniformly shape-regular simplicial partition, of the domain $\Omega$. 
We write $h_K$ and $h=\max_K h_K$ for the local and global mesh size. We further denote by 
\begin{equation*}
V_h = \{ v \in H^1(\Omega) : v|_K  \in P_k(K) \quad \forall K \in T_h\}
\end{equation*}
the standard finite element space consisting of continuous piecewise polynomial functions of degree $\le k$. 
Let $I_\tau = \{0 = t^0 < t^1 < \ldots < t^N = T\}$ denote a partition of the time interval $[0,T]$ into elements $[t^{n-1},t^n]$ of size $\tau_n = t^n-t^{n-1}$ and, as before, write $\tau=\max_n \tau_n$ for the global time step size. We denote by
\begin{equation*}
P_q(I_\tau;X) = \{v :  v|_{[t^{n-1},t^n]} \in P_q([t^{n-1},t^n];X)\} 
\end{equation*}
the space of piecewise polynomial functions in time of degree $\le q$ with values in $X$.
As approximation for the Westervelt equation \eqref{eq:westervelt} we now consider the following inexact Galerkin-Petrov Galerkin approximation of the weak formulation \eqref{eq:weak1}--\eqref{eq:weak2}. 
\begin{problem} \label{prob:1}
Find $\psi_h,p_h \in P_{q}(I_\tau;V_h) \cap H^1([0;T];V_h)$ such that 
$\psi_h(0)=\psi_{h,0}$, $p_h(0)=p_{h,0}$, for given initial values $\psi_{h,0},p_{h,0} \in V_h$,
and such that
\begin{eqnarray*}
&&\int_{t^m}^{t^n} \langle (1-2\beta p_h) \partial_t p_h, \widetilde \eta_h\rangle_h
- \alpha \langle \nabla \partial_t \psi_h, \nabla \widetilde \eta_h\rangle \, dt 
= -\int_{t^m}^{t^n} \langle \nabla \psi_h, \nabla \widetilde \eta_h\rangle \, dt \label{eq:weak1h}\\
&&-\int_{t^m}^{t^n} \langle (1-2\beta p_h) \partial_t \psi_h, \widetilde q_h \rangle_h \, dt 
= -\int_{t^m}^{t^n} \langle(1-2\beta p_h) p_h, \widetilde q_h\rangle_h \, dt. \label{eq:weak2h}\\
\notag
\end{eqnarray*}
\vspace*{-3em}

\noindent
for all $0 \le t^m \le t^n \le T$ and all $\widetilde \eta_h, \widetilde q_h \in P_{q-1}(I_\tau;V_h)$. Here $\langle u,v\rangle_h$ is a symmetric positive definite approximation for $\langle u,v \rangle$ obtained by numerical integration.
\end{problem}

Due to the inexact realization of the scalar product in some of the terms, we have to modify the discrete energy accordingly and define 
\begin{equation*}
\mathcal{E}_h(\psi_h,p_h) = \langle \tfrac{1}{2} \nabla \psi_h, \nabla \psi_h \rangle + \langle (\tfrac{1}{2} - \tfrac{2\beta}{3} p_h) p_h,p_h \rangle_h.
\end{equation*}
Note that $\mathcal{E}_h(\psi_h,p_h) = \mathcal{E}(\psi_h,p_h)$ when the scalar products are computed exactly, so this defines a natural modification of the energy on the discrete level. With similar arguments as used in Lemma~\ref{lem:energy}, we now obtain the following discrete energy identity. 
\begin{lemma} \label{lem:energyh}
Let $(\psi_h,p_h)$ denote a solution of Problem~\ref{prob:1}. Then 
one has 
\begin{equation*}
\mathcal{E}_h(\psi_h(t^n),p_h(t^n)) = \mathcal{E}_h(\phi_h(t^m),p_h(t^m))) - \alpha\int_{t^m}^{t^n} \int_\Omega |\nabla \partial_t \psi_h(s)|^2 dx \, ds, 
\end{equation*}
for all $0 \le t^m \le t^n \le T$, which is the discrete equivalent of the integral form of \eqref{eq:identity}. 
\end{lemma}
\begin{proof}
Let $u^n=u(t^n)$ denote the value of a function a time $t^n$. Then 
by the fundamental theorem of calculus and the expression of the energy derivative, we obtain
\begin{eqnarray*}
\mathcal{E}_h(\psi_h^n,p_h^n) -\mathcal{E}_h(\phi_h^m,p_h^m)   
&=&\int_{t^m}^{t^n} \frac{d}{dt} \mathcal{E}_h(\psi_h,p_h) dt \\
&=&\int_{t^m}^{t^n} \langle \nabla \psi_h,\nabla \partial_t \psi_h\rangle + \langle (1-2\beta p_h) p_h,\partial_t p_h\rangle_h \, dt.   
\end{eqnarray*}
The two terms in the second line correspond to the negative of the right hand side in Problem~\ref{prob:1} with test functions $\widetilde \eta_h = \partial_t \psi_h$ and $\widetilde q_h = \partial_t p_h$,
which directly leads to 
\begin{eqnarray*}
\mathcal{E}_h(\psi_h^n,p_h^n) -\mathcal{E}_h(\phi_h^m,p_h^m) 
&=& -\alpha\int_{t^m}^{t^n} \langle \nabla \partial_t \psi_h,\nabla \partial_t \psi_h\rangle \, dt.
\end{eqnarray*}
The assertion of the lemma now follows from the definition of the bracket $\langle \cdot,\cdot\rangle$.
\end{proof}
\begin{remark}
Let us note that, exactly in the same way as in the previous section, the discrete energy identity is a direct consequence of the particular structure of the weak formulation used in the definition of Problem~\ref{prob:1}, which adequately accounts for the underlying nonlinear discrete energy. 
\end{remark}

\section{Remarks on the implementation} \label{sec:impl}

Before we proceed to numerical tests, let us briefly comment on the  implementation of the method resulting from Problem~1. For ease of presentation, we consider  piecewise linear approximations in space and time, i.e., $k=q=1$. 
We choose the standard nodal basis for the finite elements in space and utilize the vertex rule for numerical integration in $\langle u,v\rangle_h$,
which gives rise to diagonal matrices associated with these integrals.
The system to be solved on every time step then takes the form 
{\footnotesize\small
\begin{eqnarray*}
&&D(1-2\beta p^{n+1/2}) \frac{p^{n+1}-p^n}{\tau} + \alpha K(1) \frac{\psi^{n+1}-\psi^n}{\tau} 
= -K(1) \psi^{n+1/2} \\
&&-D(1-2\beta p^{n+1/2}) \frac{\psi^{n+1}-\psi^n}{\tau} 
= -D(1-2\beta p^{n+1/2}) p^{n+1/2} - \tfrac{\beta}{6} D(p^{n+1}-p^n) (p^{n+1}-p^n)
\end{eqnarray*}}%
\noindent
with $u^{n+1/2}=\frac{1}{2}(u^n+u^{n+1})$ denoting the value at the midpoint of the time interval. Furthermore, the matrices $D(a)$, $K(b)$ represent the integrals $\langle a u,v\rangle_h$ and $\langle b \nabla u, \nabla v\rangle$.

\begin{remark}
Apart from the last term in the second equation, the time-step iteration amounts to the Gau\ss-Runge-Kutta method with $s=1$ stages and could also be interpreted as an inexact realization of the Lobatto-IIIA method with $s=2$ stages. 
Similar statements can be made for and order $q \ge 1$ in Problem~\ref{prob:1}.
Using an inexact computation of the time integrals arising on the left-hand side in Problem~\ref{prob:1} leads to the \emph{average vector field collocation methods} discussed in \cite{HairerLubichWanner}.
The inexact realization $\langle \cdot,\cdot \rangle_h$ of the scalar product in space allows to utilize mass-lumping strategies which facilitates the handling of the nonlinear terms in the numerical realization, since 
they only appear in the diagonal matrices $D(\cdot)$.
Using the considerations of \cite{Cohen02,Geevers18}, mass lumping can be achieved in principle for any order of approximation $k \ge 1$ in space. 
\end{remark}

\section{Numerical tests} \label{sec:num}

For illustration of our results, we now report about numerical tests for a simple example. We consider the Westervelt equation \eqref{eq:westervelt} on the domain $\Omega=(0,16)$ with homogeneous boundary conditions $\partial_x \psi=0$ at $\partial\Omega$. The model parameters are set to $\alpha=0$ and $\beta=0.3$, i.e., we consider a problem without dissipation. 
By Lemma~\ref{lem:energy}, the acoustic energy of the system is then preserved 
for all times. 
As initial conditions for our computational tests, we choose $\psi_0(x) = 0$ and $p_0(x)=e^{-0.2 x^2}$. Some snapshots of the numerical solution obtained with the method of Problem~\ref{prob:1} with polynomial orders $k=q=2$ are depicted in Figure~\ref{fig:wave}. 
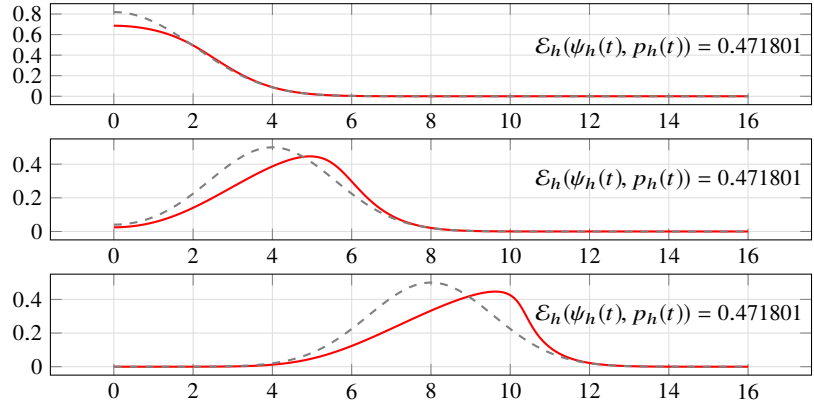
\begin{figure}[h]
\centering
\begin{tikzpicture}
\footnotesize
\begin{axis}[
width = 0.7\textwidth,
height = 0.25\textwidth,
grid=both, 
minor grid style={gray!25}, 
major grid style={gray!25},
no marks,
width=\linewidth]
\addplot[red, thick] table[x={t}, y={solT65}] {data.txt};
\addplot[gray, dashed, thick] table[x={t}, y={linsolT65}] {data.txt};
\node[above] at (14,0.3) {$\mathcal{E}_h(\psi_h(t),p_h(t)) = 0.471801$};
\end{axis}
\end{tikzpicture}

\begin{tikzpicture}
\footnotesize
\begin{axis}[
width = 0.7\textwidth,
height = 0.25\textwidth,
grid=both, 
minor grid style={gray!25}, 
major grid style={gray!25},
no marks,
width=\linewidth]
\addplot[red, thick] table[x={t}, y={solT257}] {data.txt};
\addplot[gray, dashed, thick] table[x={t}, y={linsolT257}] {data.txt};
\node[above] at (14,0.2) {$\mathcal{E}_h(\psi_h(t),p_h(t)) = 0.471801$};
\end{axis}
\end{tikzpicture}

\begin{tikzpicture}
\footnotesize
\begin{axis}[
width = 0.7\textwidth,
height = 0.25\textwidth,
grid=both, 
minor grid style={gray!25}, 
major grid style={gray!25},
no marks,
width=\linewidth]
\addplot[red, thick] table[x={t}, y={solT513}] {data.txt};
\addplot[gray, dashed, thick] table[x={t}, y={linsolT513}] {data.txt};
\node[above] at (14,0.2) {$\mathcal{E}_h(\psi_h(t),p_h(t)) = 0.471801$};
\end{axis}
\end{tikzpicture}
\caption{Solution $p_h(t)$ of the Westervelt equation with $\alpha=0, \beta=0.3$ (red) and the linear wave equation with $\alpha=\beta=0$ (black dashed) at time steps $t=1$, $t=4$, and $t=8$. 
}
\label{fig:wave}
\end{figure}
In comparison to the solution of the linear wave equation, which corresponds to \eqref{eq:westervelt} with $\alpha=\beta=0$, the presence of the nonlinear terms ($\beta=0.3$) leads to a steepening of the wave front. In the absence of viscous damping, this leads to the formation of a shock inn the long run.  
For the linear wave equation ($\beta=0$), our method coincides with the Lobatto-IIIA method and the energy is preserved exactly for both schemes. 
While the proposed method still yields exact energy preservation also in the nonlinear case ($\beta > 0$), the Lobatto-IIIA method fails to do so. Similar statements also hold for the Gau\ss-Runge-Kutta and the Newmark scheme.

From the usual error analysis of Galerkin methods \cite{Akrivis11}, we expect that the error 
$$
\mbox{err} = \max_{0 \le t_n \le T} \|p(t^n) - p_h^n\|_{h}
$$
of the method resulting from Problem~\ref{prob:1} with approximation orders $q=k$ converges with order $p=k+1$ in space and time. In Table~\ref{tab:error}, we report about the corresponding convergence rates observed in our numerical tests.
\begin{table}[h]
\centering
\begin{tabular}{|r|c|c||r|c|c|}
\hline
$h=\tau $ & err $ \times 10^{-3}$ & eoc & $h=\tau $ & err$ \times 10^{-5}$ & eoc\\
\hline
0.25    & $1.7758$ &  -   & 0.25     & $2.4964$ & -    \\
0.125   & $0.1841$ & 3.27 & 0.125    & $0.1565$ & 3.99 \\
0.0625  & $0.0131$ & 3.81 & 0.0625   & $0.0098$ & 4.00 \\
0.03125 & $0.0008$ & 4.03 & 0.03125  & $0.0006$ & 4.03 \\
\hline
\end{tabular}
\caption{Convergence rates for discrete error in the pressure at gridpoints for the nonlinear wave equation $\beta = 0.3$ (left) and the linear wave equation $\beta = 0$ (right) for comparison. }
\label{tab:error}
\end{table}
For our numerical tests, we use polynomial orders $k=q=2$ in space and time, and thus would expect third order convergence. As can be seen in Table~\ref{tab:error}, we here even observe fourth order convergence on grid-points.
This kind of super-convergence on uniform grids can be observed also for finite-difference approximations of linear wave equations \cite{CohenJoly96}.

\begin{acknowledgement}
The authors are grateful for support by the German Research Foundation (DFG) via grants TRR~146 C3, TRR~154 C4, Eg-331/1-1, and through the ``Center for Computational Engineering'' at TU Darmstadt.
\end{acknowledgement}

\end{document}